\documentclass[12pt]{amsart}
\usepackage{amsmath}
\usepackage{amsthm}
\usepackage{amsfonts}
\usepackage{amssymb}
\newtheorem{Theorem}{Theorem}[section]
\newtheorem{Proposition}[Theorem]{Proposition}
\newtheorem{Lemma}[Theorem]{Lemma}
\newtheorem{Corollary}[Theorem]{Corollary}
\theoremstyle{definition}

\theoremstyle{remark}

\numberwithin{equation}{section}

\newcommand{\Z}{{\mathbb Z}}
\newcommand{\R}{{\mathbb R}}
\newcommand{\C}{{\mathbb C}}

\newcommand{\T}{{\mathbb T}}
\newcommand{\SL}{{\textrm{\rm SL}}}
\newcommand{\PSL}{{\textrm{\rm PSL}}}

\newcommand{\tr}{{\textrm{\rm tr}\:}}

\renewcommand{\Im}{{\textrm{\rm Im}\:}}
\renewcommand{\Re}{{\textrm{\rm Re}\:}}
\newcommand{\cone}{{\textrm{\rm Cone}}}

\begin{document}

\title[Reflectionless canonical systems]{Topological properties of reflectionless canonical systems}

\author{Max Forester}

\address{Department of Mathematics\\
University of Oklahoma\\
Norman, OK 73019}

\email{mf@ou.edu}
\urladdr{www.math.ou.edu/$\sim$forester}
\author{Christian Remling}

\address{Department of Mathematics\\
University of Oklahoma\\
Norman, OK 73019}
\email{christian.remling@ou.edu}
\urladdr{www.math.ou.edu/$\sim$cremling}

\date{September 5, 2024}



\thanks{MF was partially supported by Simons Foundation award 638026.}
\begin{abstract}
We study the topological properties of spaces of reflectionless canonical systems.
In this analysis, a key role is played by a natural action of the group $\PSL(2,\R)$ on these spaces.
\end{abstract}
\maketitle
\section{Introduction}
A \textit{canonical system }is a differential equation of the form
\begin{equation}
\label{can}
Ju'(x) = -zH(x)u(x) , \quad J=\begin{pmatrix} 0 & -1 \\ 1 & 0 \end{pmatrix} ,
\end{equation}
with a locally integrable coefficient function $H(x)\in\R^{2\times 2}$, $H(x)\ge 0$, $\tr H(x)=1$.
Canonical systems define self-adjoint relations and operators on the Hilbert spaces
\[
L^2_H(I) = \left\{ f: I\to\C^2 : \int_I f^*(x)H(x)f(x)\, dx < \infty \right\} .
\]
They are of fundamental importance in spectral theory because they may be used
to realize arbitrary spectral data \cite[Theorem 5.1]{Rembook}; much of the foundational work was done by de~Branges, from a rather different
point of view \cite{dB}.

A canonical system on $x\in I=\mathbb R$ is called \textit{reflectionless }on a Borel set $A\subseteq\mathbb R$ if
\begin{equation}
\label{refless}
m_+(t)= -\overline{m_-(t)}
\end{equation}
for (Lebesgue) almost every $t\in A$. Here, $m_{\pm}$ are the \textit{Titchmarsh-Weyl }$m$ \textit{functions }of the half line problems on
$x\in [0,\infty)$ and $x\in (-\infty, 0]$, respectively. These functions are key tools in the spectral analysis of \eqref{can}; please see
\cite[Chapter 3]{Rembook} for a detailed treatment. They may be defined as
\begin{equation}
\label{defm}
m_{\pm}(z) = \pm f_{\pm}(0,z) ,
\end{equation}
with $z\in\mathbb C^+=\{ z\in\mathbb C: \Im z >0 \}$ and $f_+=u$ denoting the
(unique, up to a factor) solution $f_+\in L_H^2(0,\infty)$ of \eqref{can}, and $f_-$ similarly denotes the solution that is square integrable on the left half line.
We also use the convenient convention of identifying a vector $f=(f_1,f_2)^t\in\mathbb C^2$, $f\not= 0$, with the point $f_1/f_2\in\mathbb C^{\infty}$
of the Riemann sphere. In particular, $m_{\pm}(z)\in\mathbb C^{\infty}$, and in fact the $m$ functions are generalized \textit{Herglotz functions, }that is,
they map the upper half plane $\mathbb C^+$ holomorphically back to $\overline{\mathbb C^+}$.
Such functions have boundary values $m(t)=\lim_{y\to 0+} m(t+iy)$ at almost all $t\in\mathbb R$, and we are referring to these in \eqref{refless}.

The $m$ function is constant and real, $m_+(z)=-\tan\alpha\in\mathbb R^{\infty}$, precisely when $H(x)\equiv P_{\alpha}$ is identically equal to the
projection $P_{\alpha}$ onto the vector $v_{\alpha}=(\cos\alpha,\sin\alpha)^t$ on $x>0$. If $H(x)$ is not of this type, then $m_+(z)$ is a genuine
Herglotz function, that is, it maps $\mathbb C^+$ holomorphically back to itself. Of course, similar remarks apply to $m_-$.

Note that the degenerate canonical systems $H(x)\equiv P_{\alpha}$, $x\in\mathbb R$,
are reflectionless on $A=\mathbb R$ according to our definition \eqref{refless} since their $m$ functions are given by
$m_{\pm}(z) = \mp\tan\alpha$. Even though they look uninteresting,
they will have a role to play in what follows. We introduce the notation
\[
\mathcal Z = \{ H(x)\equiv P_{\alpha} : 0\le\alpha<\pi \}
\]
for the space of these system. We frequently refer to the $H\in\mathcal Z$ as \textit{singular }canonical systems.
Notice that we can naturally identify $\mathcal Z$ with a circle.

Reflectionless canonical systems are important because they can be thought of as the basic building blocks of arbitrary operators with
some absolutely continuous spectrum; compare \cite{BLY}, \cite[Ch.\ 7]{Rembook}. Here, we study spaces of such systems. Let's introduce
\[
\mathcal R (A) =\{ H(x): H \textrm{ is reflectionless on }A\} .
\]
We will then be interested in the (much) smaller spaces
\[
\mathcal R_0(C) = \{ H\in\mathcal R (C): \sigma(H) \subseteq C \}
\]
for closed sets $C\subseteq\mathbb R$, and also in $\mathcal R_1(C)=\mathcal R_0(C)\setminus\mathcal Z$. By definition, $\sigma(H)=\emptyset$
if $H\in\mathcal Z$ (the operator associated with such an $H$ acts on the zero Hilbert space, so there is really no spectral theory to discuss, and this is a convention),
so $\mathcal Z\subseteq\mathcal R_0(C)$ for any $C\subseteq\mathbb R$, and the spaces $\mathcal R_0(C)$ and $\mathcal R_1(C)$ differ by the collection
of singular canonical systems.

The combination of conditions used to define $\mathcal R_0(C)$ is natural. For example, these spaces occur as the sets of limit points in
generalized Denisov-Rakhmanov type theorems \cite{Den}, \cite[Theorem 1.8]{RemAnn}. Unlike the unwieldy large spaces $\mathcal R(A)$, they can
be analyzed in considerable detail.

We will only discuss the most basic case where $C$ is a \textit{finite gap set,}
that is, a union of finitely many closed intervals. For such sets, all $H\in\mathcal R_1(C)$ will satisfy $\sigma(H)=C$ because reflectionless on $C$
operators have this set in their absolutely continuous spectrum. This is well known and will also be an easy by-product of what we'll do in Section 2.

The analysis naturally splits into three cases: $C$ can have two, one, or no unbounded components. Typical (simple) examples are $C=(-\infty,-1]\cup [1,\infty)$,
$C=[0,\infty)$, and $C=[-2,2]$, respectively. It will turn out that the last case is considerably more involved and richer than the first two.

Such sets also occur as the spectra of more classical differential and difference operators,
and then these three cases correspond to Dirac, Schr{\"o}dinger, and Jacobi operators, in this order. We will analyze how these specialized operators
sit inside the larger spaces $\mathcal R_0(C)$.

This issue is related to a natural group action on $\mathcal R_0(C)$ and $\mathcal R_1(C)$.
Recall that $\PSL(2,\R)$ is the quotient of $\SL(2,\R)=\{ A\in\R^{2\times 2}: \det A=1\}$ by $\{\pm 1\}$.
This group acts on the upper half plane by linear fractional transformations.
If we think of $\C^+$ as a subset of projective space and thus again (as in \eqref{defm})
represent points $z$ by vectors $v\in\C^2$, $v_1/v_2=z$, then an $A\in\PSL(2,\R)$ simply acts as $Av$,
that is, we apply the matrix $A$ to the vector $v$.

We then also obtain an action of $\PSL(2,\R)$ on canonical systems, by letting group
elements act on the half line $m$ functions pointwise, as follows:
\[
\pm m_{\pm}(z; A\cdot H) = A\cdot (\pm m_{\pm}(z; H))
\]
This action preserves spectra as well as the property of being reflectionless on a set \cite[Theorems 7.2, 7.9]{Rembook}. Moreover, $\mathcal Z$ is
clearly invariant under the action and thus so are $\mathcal R_0(C)$, $\mathcal R_1(C)$.

We will analyze this group action in all three settings, corresponding to sets $C$ of Dirac, Schr{\"o}dinger, and Jacobi types. In the first two cases, it is mostly
an additional gadget that is available if desired. However, in the Jacobi case, the group action becomes a valuable tool that will allow us
to give a very elegant treatment of an obstinate technical issue related to the presence of non-trivial fibers. This issue could be
analyzed directly, and that was in fact done in \cite{Remtop} for a different version of the same problem, but that analysis becomes very tedious.

We refer the reader to Sections 3--5 for precise formulations of our results, but let's at least attempt a quick summary of what we will prove here: the spaces
$\mathcal R_1(C)$, endowed with a natural metric that can be defined on arbitrary canonical systems, are homeomorphic
to a product of a disk with a torus $\mathbb D\times\mathbb T^N$, with each component of $C^c$ contributing one circle $\mathbb S^1$ to $\T^N$.
See Corollaries \ref{C3.1}(a), \ref{C4.1}(a), \ref{C5.1}.
Given a suitable parametrization of $\mathcal R_1(C)$, which we'll review in Section 2,
it will be rather straightforward to establish this in the Dirac and Schr{\"o}dinger cases; in the Jacobi case, which we'll deal with in Section 5,
the analysis becomes much more intricate because the two unbounded components of $C^c$ meet at the point $\infty$, while components
are always well separated in the other cases.

Moving on to the discussion of the group actions, we will in all cases identify $\mathcal R_1(C)$ with a product of the acting group and a space of
suitably chosen representatives of the orbits (= Theorems \ref{T3.1}, \ref{T4.1}, \ref{T5.2}).
This will be easiest in the Schr{\"o}dinger case, where we will show that $\mathcal R_1(C)\cong \PSL(2,\R)\times\mathcal S(C)$,
with the second factor denoting the Schr{\"o}dinger operators in $\mathcal R_1(C)$. The Dirac case poses no great challenges either, but in the Jacobi
case, we will have to make a very careful choice of representatives.

From this product structure we can then also deduce that the orbit space is homeomorphic to a torus $\T^{N-1}$
(= Corollaries \ref{C3.1}(b), \ref{C4.1}(b), \ref{C5.1}).

Finally, the space $\mathcal R_0(C)\supseteq\mathcal R_1(C)$ is a compactification of $\mathcal R_1(C)$, obtained by adding a circle. When $\mathcal R_1(C)$
is three-dimensional, we obtain the $3$ sphere $\mathcal R_0(C)\cong\mathbb S^3$,
and in all other cases, $\mathcal R_0(C)$ is not locally Euclidean at the points of the extra circle. See Theorems \ref{T3.2}, \ref{T4.2}, \ref{T5.3}.
\section{Parametrization of reflectionless canonical systems}
We adapt the method that was discussed in detail in \cite{PoltRem2,Remtop} to canonical systems; the original version dealt with Jacobi matrices.
These ideas go back to at least \cite{Craig}. We will focus mostly on the new aspects and refer the reader to \cite{PoltRem2,Remtop} for further details
on some of the more routine steps.

Given an $H\in\mathcal R(C)$, let
\begin{equation}
\label{h(z)}
h(z) = m_+(z)+m_-(z) .
\end{equation}
We also assume for now that $H\notin\mathcal Z$, so at least one of $m_{\pm}$ is a genuine Herglotz function, not a constant $a\in\R^{\infty}$.
But then in fact both of $m_{\pm}$ are genuine Herglotz functions (or else $H$ could not be reflectionless), and thus so is $h(z)$.

Next, from the condition that $H$ is reflectionless on $C$, we obtain $\Re h(t) =0$ for almost every $t\in C$. Thus the \textit{Krein function}
\[
\xi(t) = \frac{1}{\pi} \lim_{y\to 0+} \Im \log h(t+iy)
\]
satisfies $\xi(t)=1/2$ almost everywhere on $t\in C$. Recall here that a Herglotz function $F(z)$ has a holomorphic logarithm $\log F(z)$,
which is a Herglotz function itself if we take the logarithm with imaginary part in $(0,\pi)$. Moreover, since $\Im\log F(z)$ is bounded, the measure from
the Herglotz representation of $\log F(z)$ is purely absolutely continuous. This gives us an alternative interpretation of $\xi$ as the density
of the measure representing $\log h(z)$. In particular, we can recover $\log h(z)$ or, equivalently, $h(z)$ itself, from $\xi$, up to a constant.
This can be done explicitly, using the \textit{Herglotz representation }formula for $\log h(z)$; this is also sometimes referred to as the
exponential Herglotz representation of $h(z)$. We have
\begin{equation}
\label{2.3}
h(z) = D \exp \left( \int_{-\infty}^{\infty} \left( \frac{1}{t-z} - \frac{t}{t^2+1}\right) \xi(t)\, dt \right) \equiv D h_0(z) ,
\end{equation}
for some $D>0$.

If we only knew that $H\in\mathcal R(C)$, then any measurable $0\le\xi\le 1$ satisfying $\xi =1/2$ on $C$ would be a possible Krein function.
We are interested in the much stronger condition $H\in\mathcal R_1(C)$, that is, $H$ is not only reflectionless on $C$,
but we also assumed that there is no spectrum outside this set.
This imposes strong additional restrictions on $\xi$. To derive these, consider the Herglotz function $g=-1/h$
and write down its Herglotz representation. We have
\[
\frac{-1}{h(z)} = a+bz + \int_{-\infty}^{\infty} \left( \frac{1}{t-z} - \frac{t}{t^2+1}\right) d\rho(t)
\]
for some Borel measure $\rho$ satisfying $\int d\rho(t)/(1+t^2)<\infty$ and $a\in\R$, $b\ge 0$. In fact, $\rho$ is a spectral measure of $H$,
and this follows from the usual way of setting up a spectral representation of the whole line problem; see \cite[eqn.\ (3.17)]{Rembook}.

In particular, since by assumption $\sigma(H)\subseteq C$, we have $\rho(C^c)=0$, and thus
the function $g=-1/h$ has a holomorphic continuation through each component $(c,d)\subseteq C^c$. Moreover, $g(x)\in\mathbb R$ and $g'(x)>0$ there.
It follows that $g(x)$ changes its sign at most once on each such interval $(c,d)$, and if there is a sign change as we increase $x$,
it can only be from negative to positive values.

Let's rephrase this in terms of the Krein function $\xi$ of $h(z)$: on each component $(c,d)$ of $C^c$ (``gap''), the Krein function is of the form
\[
\xi(t) = \chi_{(\mu ,d)}(t) ,
\]
for some $c\le\mu\le d$. Since $\xi=1/2$ on $C$, as we observed earlier, the parameters $\mu_j$, one for each gap, give a complete description of $\xi(t)$
and thus also of $h(z)$, up to the multiplicative constant $D>0$, which we keep as an additional parameter.

This is actually a somewhat subtle technical point. Depending on the precise shape of the set $C$, we may want to incorporate a certain $\mu$ dependent factor
in $D$ from \eqref{2.3} and redefine $h_0$ accordingly. This will ensure that the singular canonical systems $H\in\mathcal Z$ fit nicely into our parameter space.
For now, we focus on the simplest case of a $C$ with no unbounded components in its complement, for which these issues are absent.
(We will deal with them when they reappear, in Sections 4, 5.)
So for the remainder of this section, we assume that
\begin{equation}
\label{CDirac}
C = \mathbb R \setminus \bigcup_{j=1}^N (c_j,d_j), \quad\quad c_1<d_1<c_2<\ldots < d_N .
\end{equation}
This will allow us to see the whole procedure in its simplest form, without currently unnecessary technical complications.

The integral from \eqref{2.3} can now be done explicitly. We obtain
\begin{equation}
\label{2.4}
h_0(z) = 2i \prod_{j=1}^N \frac{\sqrt{(c_j-z)(d_j-z)}}{\mu_j-z} ;
\end{equation}
the square root is determined by the requirement that $\Im h_0(z)>0$. Alternatively, we can prove \eqref{2.4} by confirming that the
right-hand side defines a Herglotz function that has the correct arguments (= Krein function) on the real line. As we already observed in the context of
\eqref{2.3}, that still leaves a multiplicative constant undetermined. Our choice of a factor of $2$ is natural because then $h_0$ corresponds to a
classical Dirac operator, but this detail is actually irrelevant for the purposes of this section.

We now return to the basic issue of parametrizing $\mathcal R_1(C)$. So far, we have introduced the parameters $\mu_j\in [c_j,d_j]$, $D>0$,
but these only determine $h(z)$, and this function does not normally determine the canonical system $H$, except in very specialized situations. Rather, the pairs
$(m_-,m_+)$ of half line $m$ functions are in one-to-one correspondence to the whole line canonical systems $H(x)$, $x\in\mathbb R$ \cite[Ch.\ 5]{Rembook}.
So we must return to \eqref{h(z)} and analyze how $h(z)$ from \eqref{2.4} can be split into two half line $m$ functions $m_{\pm}$.

We now abandon the exponential Herglotz representation of $h(z)$ and its Krein function and use the traditional representations instead. There is a small choice to make
about how exactly to incorporate $D$ from \eqref{2.3}, and it will be convenient to proceed as follows. Write
\begin{align}
\label{2.31}
h_0(z) & = A + \int_{\mathbb R^{\infty}} \frac{1+tz}{t-z}\, d\nu(t) ,\\
\label{2.5}
m_{\pm}(z) & = A_{\pm} + D \int_{\mathbb R^{\infty}} \frac{1+tz}{t-z}\, d\nu_{\pm}(t) .
\end{align}
The data for $h_0(z)$, namely $A\in\mathbb R$ and the finite measure $\nu$, are in principle available to us
since we have $h_0(z)$ from the $\mu_j$ via \eqref{2.4}. We are using a slightly different version of the Herglotz representation formula here: instead of writing
\[
h_0(z) = A +Bz + \int_{-\infty}^{\infty} \left( \frac{1}{t-z} -\frac{t}{t^2+1} \right)\, d\rho(t) , \quad\int_{-\infty}^{\infty} \frac{d\rho(t)}{t^2+1}<\infty ,
\]
as we did above, we now represent $h_0$ by \eqref{2.31}. We can easily go back and forth between these two versions. For example, we obtain \eqref{2.31}
by letting $d\nu=B\delta_{\infty}+d\rho(t)/(t^2+1)$. Frequently \eqref{2.31} is more convenient to work with
because $\nu$ is a finite Borel measure on the compact space $\R^{\infty}$.

Our task is to find all $A_{\pm}$, $\nu_{\pm}$ that make \eqref{h(z)} happen and produce an $H\in\mathcal R_1(C)$.
Of course, by the uniqueness of Herglotz representations, what \eqref{h(z)} is asking for is simply that $A_-+A_+=DA$, $\nu_-+\nu_+=\nu$.

To split the measure $\nu$ in this fashion, we remind ourselves of its basic structure: For the finite gap sets $C$ considered here,
it's easy to see that $\nu$ is purely absolutely continuous on $C$, with density
\[
\chi_C(t)\, d\nu(t)=\chi_C(t)\,\frac{\Im h_0(t)\, dt}{\pi(1+t^2)} .
\]
There is nothing to choose about this part when we split $\nu$:
Condition \eqref{refless} forces $\Im m_-=\Im m_+$ on $C$, so $\chi_C\, d\nu_{\pm} = (1/2)\chi_C\, d\nu$.

On $C^c$, the measure $\nu$ has a point mass at each jump $\mu_j\not= a_j, b_j$, and $\nu$ does not give weight to other sets.
This follows most conveniently from a standard criterion for point masses in terms of the Krein function \cite[pg.\ 201]{MPut}.
In other words,
\[
d\nu (t) = \sum_j w_j \delta_{\mu_j} + \chi_C(t)\frac{h_0(t)\, dt}{\pi (1+t^2)} ,
\]
with $w_j>0$ precisely when $\mu_j\not= a_j, b_j$. These point masses may be computed as
\begin{equation}
\label{2.66}
w_j = \frac{-i}{1+\mu_j^2}\lim_{y\to 0+} yh_0(\mu_j+iy) .
\end{equation}

Now for any given point mass $w\delta_{\mu}$, any of the choices $\nu_+=s w\delta_{\mu}+\ldots$, $0\le s\le 1$ will satisfy
\eqref{h(z)}. However, a closer look reveals that only $s=0,1$ are compatible with the requirement $\sigma(H)\subseteq C$; a choice
of $s\in (0,1)$ would produce an eigenvalue at $\mu\notin C$. We will not give the details of this argument. Instead, we refer the reader to
\cite{PoltRem2} and \cite[Section 3.7]{Rembook}.

What we just discussed is most usefully stated in terms of the half line $m$ function $m_+$. We have the formula
\begin{equation}
\label{2.11}
m_+(z) = A_++D \left( \frac{1}{2}(h_0(z)-A) + \sum_{j=1}^N \frac{2s_j-1}{2} w_j \frac{1+\mu_j z}{\mu_j-z} \right) .
\end{equation}
A \textit{reflectionless }canonical system is already determined by $m_+$ only. See \cite[Theorem 7.9(b)]{Rembook}, but perhaps also observe
that the statement is unsurprising since \eqref{refless} provides enough information about $m_-$ to reconstruct this function from $m_+$.
So for the purposes of obtaining a parametrization, we can temporarily forget about the intermediate steps and summarize by saying that
the parameters $\mu_j\in [c_j,d_j]$, $s_j=0,1$, $A_+\in\R$, $D>0$ determine $m_+$ and thus also the canonical system $H\in\mathcal R_1(C)$.
Note that $A\in\R$ and the $w_j\ge 0$ are not independent parameters. Rather, we find these quantities from the $\mu_j$ via $h_0(z)$ and \eqref{2.31}, \eqref{2.66}.

We are ready to put on the finishing touches.
We combine $s_j\in \{ 0,1\}$ and $\mu_j\in [c_j,d_j]$ into a single parameter $\widehat{\mu}_j = (\mu_j,s_j)$.
Recall that $s_j$ becomes irrelevant when $\mu_j=c_j$ or $\mu_j=d_j$ because then $\nu$ does not have a point mass at $\mu_j$ that
needs to assigned to $\nu_-$ or $\nu_+$ or, to say the same thing more formally, because $w_j=0$ then. Thus we can naturally
think of each $\widehat{\mu}_j$ as coming
from a circle (two copies of $[c_j,d_j]$, glued together at the endpoints). Not only is this convenient for the book-keeping, but, much more importantly,
it will also turn out that the topology suggested by this procedure is the right one. Let's review one more time how these
parameters work in \eqref{2.11}: The $\widehat{\mu}_j$ determine the function in parentheses as well as $A\in\mathbb R$; actually,
this latter quantity only depends on the $\mu_j$ and not on the $s_j$, as does $h_0(z)$. To completely specify $m_+$, we then need the additional parameters
$D>0$ and $A_+\in\mathbb R$.

Let's add some precision to our (mostly implicit, so far) claims. Given an $H\in\mathcal R_1(C)$, we have introduced parameters
$\widehat{\mu}_j$, $D>0$, $A_+\in\mathbb R$. These are determined by $H$, and, conversely, they may be used recover $H\in\mathcal R_1(C)$ via
\eqref{2.11}.
In fact, any such set of parameters corresponds to a unique $H\in\mathcal R_1(C)$. This last part we did not discuss explicitly, but it is clear how to proceed:
one simply constructs $m_{\pm}$ from the parameters, using the recipes given, and then checks that these correspond to a unique $H\in\mathcal R_1(C)$.
We have set up a bijection between $\mathcal R_1(C)$ and the parameter space $\{(A_+,D,\widehat{\mu}_j)\}$.

Finally, we introduce topologies. The space of trace normed canonical systems becomes a compact metric space when endowed with a natural metric,
which is discussed in detail in \cite[Section 5.2]{Rembook}. More importantly for us here, the one-to-one correspondence $H\mapsto (m_-,m_+)$
between canonical systems and pairs of generalized Herglotz functions
becomes a homeomorphism if we equip the space of Herglotz functions with the metric
\[
d(F,G) = \max_{|z-2i|\le 1} \delta( F(z), G(z)) ;
\]
see \cite[Corollary 5.8]{Rembook}.
Convergence in $d$ is equivalent to locally uniform convergence with respect to the spherical metric $\delta$, with $\mathbb C^+$ thought of as a subset of the
Riemann sphere $\C^{\infty}\cong\mathbb S^2$.
Note that thanks to a normal families argument, a single compact set $|z-2i|\le 1$ with non-empty interior is sufficient to control all the others.
So it is not necessary to exhaust $\C^+$ by a sequence of increasing compact sets, which would otherwise have been the standard procedure.

We already mentioned the key fact that reflectionless systems are determined
by their half line restrictions \cite[Theorem 7.9(b)]{Rembook}, and $m_-$ can be reconstructed from $m_+$. Moreover,
the induced map $m_+\mapsto m_-$, defined on the compact space of $m$ functions $m_+(z;H)$ with $H\in\mathcal R(C)$, is continuous.
As a consequence, we can also measure the distance between
two \textit{reflectionless }systems by only computing $d(m_+^{(1)}, m_+^{(2)})$, and we still obtain the same topology.

Recall that the singular canonical systems $H\in\mathcal Z$ correspond to the constant $m$ functions $m_+(z)=-m_-(z)\equiv a\in\mathbb R^{\infty}$,
so $\mathcal Z$ is homeomorphic to a circle. To conveniently attach this circle to our parameter space, we combine $Z=A_++iD\in\mathbb C^+$
into a single parameter.
\begin{Proposition}
\label{P2.1}
Let $H_n\in\mathcal R_1(C)$, with parameters $Z_n$, $\widehat{\mu}^{(n)}_j$. Denote the singular canonical system with $\pm m_{\pm}(z)\equiv a$,
$a\in\mathbb R^{\infty}$, by $K_a\in\mathcal Z$. Then $H_n\to K_a$ if and only if $\delta(Z_n,a)\to 0$.
\end{Proposition}
\begin{proof}
From \eqref{2.5}, we have $m_+(i)=A_++iD\nu_+(\mathbb R^{\infty})$. Recall that $\nu_+$ only depends on the parameters $\widehat{\mu}_j$ and not
on $A_+,D$. Moreover, \eqref{2.4} implies that we have uniform bounds of the form
\begin{equation}
\label{2.41}
0<c_1\le \nu_+(\R^{\infty})\le \nu(\mathbb R^{\infty})\le c_2 ,
\end{equation}
with $c_1,c_2$ independent of the $\widehat{\mu}_j$. To obtain the lower bound, we can simply estimate $\nu_+$ by its absolutely continuous part,
which is supported by $C$ and has density (essentially) $h_0(x)$ there.

Assume first that $\delta(Z_n, a)\to 0$, $a\in\mathbb R$. Equivalently,
$A_+^{(n)}\to a$, $D_n\to 0$.
The space $\mathcal R_0(C)$ is compact
\cite[Proposition 1.4]{PoltRem2}, \cite[Theorem 7.11]{Rembook}. Thus $H_n$ always converges along suitable subsequences, to $H\in\mathcal R_0(C)$, say.
Using the upper bound from \eqref{2.41},
we conclude that the $m$ function of the limit point $H$ satisfies $m_+(i;H)=a$, but the only such Herglotz function is $m_+(z)\equiv a$, so $H=K_a$.
We have in fact shown that $K_a$ is the only possible limit point of $H_n$,
and thus the full sequence also converges to this limit, without the need of passing to a subsequence.

The same argument handles the case $\delta (Z_n,\infty)\to 0$. This assumption implies that $|m_+^{(n)}(i)|\to\infty$
because now $|A_+^{(n)}|+D_n\to\infty$, and we can use
the lower bound from \eqref{2.41} for those $n$ (if any) for which $|A_+^{(n)}|$ is not large.

Conversely, if $H_n\to K_a$, so $m_+^{(n)}(z)\to a$ locally uniformly, then we can similarly find a subsequence (written as the original sequence, for convenience)
such that $\delta (Z_n, Z)\to 0$, for some $Z\in\overline{\mathbb C^+}$. If $Z\not= a$, then $m_+^{(n)}(i)$ can not converge to $a$. This we see
by distinguishing the two cases $Z\in\C^+$, $Z\in\R^{\infty}$ and arguing as in the first part of this proof; in particular, we again use \eqref{2.41} when required.

As above, this argument has shown that every subsequence has a sub-subsequence along which $Z_n\to a$, so the original sequence converges
to this limit, as claimed.
\end{proof}
\section{Dirac case: two unbounded components}
We continue our discussion of spectra $C$ of the form \eqref{CDirac}. We call this the \textit{Dirac case }because $\mathcal R_1(C)$ for such $C$ contains
Dirac operators
\[
Dy = Jy'(x)+W(x)y(x) ,
\]
acting on $y\in L^2(\mathbb R; \mathbb C^2)$. More precisely, a variation of constants procedure lets us rewrite these Dirac equations $Dy=zy$ as canonical
systems, and this accounts for some of the members of $\mathcal R_1(C)$; see also \cite[Section 6.4]{Rembook}.

We denote by $\mathcal D$ the collection of canonical systems that are Dirac operators, in this sense.
It will also be convenient to abbreviate
\[
\mathcal D(C) = \mathcal R_1(C)\cap\mathcal D .
\]
Similar notation will be used in the other cases in Sections 4, 5 below.

If this all sounds a bit vague since we didn't give the details of the procedure
that rewrites a Dirac equation as a canonical system, then we refer the reader to a precise criterion for an $H\in\mathcal R_1(C)$ to belong to $\mathcal D$,
in terms of its $m$ functions $m_{\pm}(z;H)$, that will be given below.

Given the work of the previous section, it will now be rather straightforward to analyze the topology of $\mathcal R_1(C)$.
We denote the open unit disk by $\mathbb D =\{ z\in\C: |z|<1\}$ and the $N$-dimensional torus by $\mathbb T^N=\{ (z_1,\ldots , z_N): |z_j|=1\}$.

We also right away consider the action of the group $\textrm{PSL}(2,\mathbb R)$ on $\mathcal R_1(C)$ though this is actually not needed if we only want
to understand the topology of $\mathcal R_1(C)$. The subgroup $\textrm{\rm SO}(2)/\{\pm 1\}$ maps $\mathcal D$ back to itself;
this will be clear from the discussion below since these group elements fix $z=i$. Thus there will be many Dirac operators in an orbit when $N\ge 1$,
and since our main use of these groups is to move us around to a specialized operator, the following subgroup is more useful here:
\begin{equation}
\label{3.5}
G = \left\{ \begin{pmatrix} c & a/c \\ 0 & 1/c \end{pmatrix} : c>0, a\in\mathbb R \right\}
\end{equation}
\begin{Theorem}
\label{T3.1}
The action of $G$ on $\mathcal R_1(C)$ is fixed point free,
and every orbit $\{ g\cdot H: g\in G\}$ contains a unique Dirac operator $H_1\in\mathcal D$.

The ($G$ equivariant) map $G\times\mathcal D(C)\cong\mathcal R_1(C)$, $(g, H_1) \mapsto g\cdot H_1$, is a homeomorphism.
\end{Theorem}
Here we of course give $G\subseteq\mathbb R^{2\times 2}$ its natural (subspace) topology. Observe that then $G\cong \mathbb D$. In particular,
since the topology of $\mathcal D(C)\cong\T^N$ can be found easily, using the material of the previous section,
we also obtain the topology of the original space.
\begin{Corollary}
\label{C3.1}
(a) $\mathcal R_1(C)$ is homeomorphic to $\mathbb D\times \mathbb T^N$.

(b) The map $\mathcal R_1(C)/G\to\mathcal D(C)$ that sends an orbit to its unique representative in $\mathcal D(C)$ is a homeomorphism,
and $\mathcal D(C)\cong\T^N$.
\end{Corollary}
Theorem \ref{T3.1} also delivers a homeomorphism between $\mathcal R_1(C)$ and $\mathbb D\times \mathbb T^N$.
As we'll see in the proof, this is closely related to but not identical
with the map from the previous section that sends an $H\in\mathcal R_1(C)$ to its
parameters $A_+, D, \widehat{\mu}_j$, and this latter map would perhaps have been the most natural choice.
\begin{proof}[Proof of Theorem \ref{T3.1}]
It was natural to list the statements in the order given, but in the argument, we will actually start with Corollary \ref{C3.1}(a).

The key tool will of course be the parametrization that was discussed in the previous section and the associated map that
sends an $H\in\mathcal R_1(C)$ to its parameters $Z=A_++iD$, $\widehat{\mu}_j$. As already discussed, we can think of $\widehat{\mu}_j=(\mu_j,s_j)$
as coming from a circle $\mathbb S^1=\{ z: |z|=1\}$ by mapping $z=e^{i\pi t}$, $-1\le t< 1$, to
\[
f(z) = \begin{cases} (c_j + t(d_j-c_j) , 1) & 0 < t < 1 \\
(c_j -t(d_j-c_j) , 0) & -1 < t < 0 \\
c_j & t=0 \\
d_j & t=-1
\end{cases} .
\]
Recall again that $s_j=0,1$ becomes irrelevant when $\mu_j=c_j$ or $=d_j$.
Similarly, we can of course identify $\mathbb C^+$ with $\mathbb D$, for example via the Cayley transform.

It will be technically convenient to first discuss the spaces $\mathcal R_0(C)$, which have the advantage of being compact,
and we'll need this extension anyway later on.
So we now consider the map $F:\overline{\mathbb D}\times\mathbb T^N\to\mathcal R_0(C)$. On $\mathbb D\times\mathbb T^N$,
it acts as just described: we interpret a point $(w,z)\in\mathbb D\times\T^N$ as a point $(Z,\widehat{\mu}_j)$ in parameter space,
using the maps just set up, and then we map this to a unique $H=F(w,z)\in\mathcal R_1(C)$. We then extend $F$ to the larger space
$\overline{\mathbb D}\times\T^N$ by identifying in the same way $(w,z)$, $|w|=1$, with the parameters $(a,\widehat{\mu}_j)$, $a\in\R^{\infty}$,
and then we send this simply to $F(w,z)=K_a\in\mathcal Z$, the singular canonical system with $m_+\equiv a$.
\begin{Lemma}
\label{L3.1}
The map $F$ induces a homeomorphism
\[
F_1: \overline{\mathbb D}\times\T^N/\!\!\sim\;\to\mathcal R_0(C) .
\]
The first space is the quotient space by the equivalence relation
\[
(w,z)\sim (w',z')\iff |w|=|w'|=1, \quad w=w' .
\]
\end{Lemma}
\begin{proof}
Notice first of all that $F$ is constant on equivalence classes, so we do obtain a well defined map $F_1$ on the quotient. In fact, it is clear
that $F_1$ is bijective, and since we are mapping
between compact metric spaces, it suffices to check continuity in one direction. (But it would also not be hard to check continuity of both $F_1$ and its inverse
separately.) We focus on $F_1: \overline{\mathbb D}\times\T^N/\!\!\sim\;\to\mathcal R_0(C)$ itself. This map was induced by
$F:\overline{\mathbb D}\times\T^N\to\mathcal R_0(C)$, so its continuity is equivalent to the continuity of $F$. At a point $(w,z)\in\mathbb D\times\T^N$,
what we need can be rephrased as the continuous dependence of an $H\in\mathcal R_1(C)$ on its parameters, and this information is easily extracted from the formulae
of Section 2, especially \eqref{2.4}, \eqref{2.11}. We also need \eqref{2.66}, which will guarantee that $w_j=w_j(\{\mu_k \})\to 0$ if
$\mu_j\to c_j$ or $\mu_j\to d_j$. This in turn makes sure that changing a $s_j$ will only have a small effect if the corresponding $\mu_j$ is close
to an endpoint of its gap.

If $|w|=1$, then the continuity of $F$ at $(w,z)$ follows at once from Proposition \ref{P2.1}.
\end{proof}
Since this quotient space contains $\mathbb D\times\mathbb T^N$ as a subspace, Lemma \ref{L3.1} also establishes the continuity of the map
and its inverse between this space and $\mathcal R_1(C)$, and we have now proved Corollary \ref{C3.1}(a).

Next, notice that if $g\in G$ is as in \eqref{3.5}, then the induced linear fractional transformation is given by
$g\cdot w = c^2 w+ a$. This shows, first of all, that the action on $\mathcal R_1(C)$ is fixed point free.
Indeed, if $g\cdot m_+(z)=m_+(z)$, then we can specialize to $z=i$, say, and since $\Im m_+(i)>0$, we deduce that $c=1$, $a=0$, that is, $g=1$.

Furthermore, a $G$ orbit $\{ g\cdot m_+\}$
contains exactly those $m_+$ that assume all possible values of the parameters $A_+,D$ while the $\widehat{\mu}_j$ are not moved by
the action of $G$.
A look at \eqref{2.4} and \eqref{2.11} shows that $m_{\pm}(z)$ are holomorphic near $z=\infty$, and in this situation, the $m$ functions
correspond to a Dirac operator precisely when $m_{\pm}(\infty)=i$ \cite{LevSar}.
It is now straightforward to check that indeed each $G$ orbit contains a unique $H\in\mathcal D$.

Obviously, the map $(g,H)\mapsto g\cdot H$ is continuous. We just showed that it is also bijective,
and both spaces involved are manifolds, so the continuity of the inverse follows from invariance of domain \cite[Proposition IV.7.4]{Dold}.
Here, we use the machinery of Section 2 in a simplified version to identify the topology of $\mathcal D(C)\cong\T^N$.
More specifically, $\mathcal D(C)$ is parametrized by just the $\widehat{\mu}_j$. For an $H\in\mathcal D(C)$, we have $D=1$,
and $A_+$ is also determined by the $\widehat{\mu}_j$ through the requirement that $m_+(\infty)=i$. So we have a bijection
between the compact metric spaces $\mathcal D(C)$ and $\T^N$, and inspection of the formulae of Section 2 will confirm that
this map is a homeomorphism. See also the corresponding discussion in the proof of Lemma \ref{L3.1}.
\end{proof}
\begin{Theorem}
\label{T3.2}
(a) $\mathcal R_0(C)$ is homeomorphic to $\mathbb S^3$ if $N=1$.

(b) $\mathcal R_0(C)$ is not a manifold if $N\ge 2$. More precisely, a point $H\in\mathcal R_0(C)$ has a locally Euclidean neighborhood
if and only if $H\in\mathcal R_1(C)$, or, equivalently, if and only if $H\notin\mathcal Z$.
\end{Theorem}
When $N=0$, we have $\mathcal R_0(C)\cong\overline{\mathbb D}$, so unlike in the case $N\ge 2$, this is still a manifold with
boundary. Compare also \cite[Theorem 7.19]{Rembook}.
\begin{proof}
(a) In this case, Lemma \ref{L3.1} says that $\mathcal R_0(C)$ is homeomorphic to the quotient of $\overline{\mathbb D}\times\mathbb S^1$ by the equivalence
relation $(e^{i\alpha},e^{i\beta})\sim (e^{i\alpha},e^{i\beta'})$. Think of $\mathbb S^3$ as the subset $\{(w,z):|w|^2+|z|^2=1\}\subseteq\mathbb C^2$,
and then consider the map $f:\overline{\mathbb D}\times\mathbb S^1\to\mathbb S^3$,
\[
f(re^{i\alpha},e^{i\beta})=\left( re^{i\alpha}, \sqrt{1-r^2}e^{i\beta}\right) .
\]
It is easy to verify that $f$ induces a homeomorphism between the quotient and $\mathbb S^3$.

(b) By Lemma \ref{L3.1} each point $H \in \mathcal R_0(C)$ is represented by a pair $(w,z) \in  \overline{\mathbb D}\times\T^N$. We
must show that $H$ admits a locally Euclidean neighborhood if and only if $w \not\in \partial\mathbb D$. Note that the manifold $\mathbb D\times\T^N$ embeds into
$\mathcal R_0(C)$ as an open set under the quotient map from $\overline{\mathbb D}\times\T^N$, so one direction is clear.

For the other direction, if $w \in \partial \mathbb D$, then $H$ has a neighborhood $U$ homeomorphic to $(-1,1)\times\cone(\T^N)$ in which $H$ has coordinates
$(0,\ast)$. Here $\cone(\T^N)$ denotes the open cone on $\T^N$ with cone point $\ast$. That is,
\[
\cone(\T^N) = [0,1)\times\T^N/\!\!\sim
\]
with $\sim$ identifying all of $\{ 0\}\times\T^N$ to a single point $\ast$.

We compute the local homology groups of $\mathcal R_0(C)$ at $H$ as follows:
\begin{align*}
H_i (U, U-\{H\}) &\cong H_{i-1}(\cone(\T^N), \cone(\T^N) - \{\ast\}) \\
&\cong \widetilde{H}_{i-2}(\cone(\T^N) - \{\ast\}) \\
&\cong \widetilde{H}_{i-2}(\T^N).
\end{align*}
The first and second isomorphisms follow from \cite[IV(3.14)]{Dold} and \cite[IV(3.12)]{Dold} respectively; the third is induced by a deformation retraction of
$\cone(\T^N) - \{\ast\}$ onto $\T^N$.

Now note that the reduced homology groups $\widetilde{H}_{i-2}(\T^N)$ are non-zero throughout the range $1 \leq i-2 \leq N$, and hence are non-zero for two or
more values of $i$ when $N\geq 2$. By contrast, the local homology at any point of an $n$-dimensional manifold is non-zero only when $i=n$.
\end{proof}
We could also let $G$ act on $\mathcal R_0(C)$, but this modification
is not particularly interesting; for example, the orbit space is not a Hausdorff space.
\section{Schr{\"o}dinger case: one unbounded component}
We now assume that $C$ is of the form
\[
C = \bigcup_{j=1}^N[d_{j-1},c_j] \cup [d_N,\infty) , \quad\quad d_0<c_1< \ldots < d_N , \quad N\ge 0 .
\]
In other words, $C$ again has the gaps $(c_j,d_j)$, $j=1,\ldots ,N$. In addition, there is the unbounded gap $(-\infty,d_0)$.

As already indicated in the title of this section, such spectra occur for Schr{\"o}dinger operators
\[
Sy = -y''(x) + V(x)y(x) ,
\]
acting on $y\in L^2(\mathbb R)$. For example, if $C=[0,\infty)$, then the Schr{\"o}dinger operator with $V\equiv 0$ lies in $\mathcal R_1(C)$.
Again, Schr{\"o}dinger equations can be rewritten as canonical systems by a variation of constants procedure. See \cite[Section 1.3]{Rembook}
for details on this. We assume in the sequel that the exact same procedure as in that source is used, which amounts to imposing specifically Dirchlet boundary conditions
on the half line Schr{\"o}dinger equations at $x=0$ when computing half line $m$ functions. In terms of the associated canonical system,
this means that $H(0)=P_{e_2}=\bigl( \begin{smallmatrix} 0&0\\0&1\end{smallmatrix}\bigr)$ (this condition is meaningful because the coefficient
function $H(x)$ corresponding to a Schr{\"o}dinger equation is continuous). This is a detail that is essentially arbitrary and
could have been handled differently, but without affecting the general nature of the results below, and some choice has to be made.
We write $\mathcal S$ for the collection of canonical systems that are, in this sense, equivalent to a Schr{\"o}dinger equation.

Return to \eqref{2.4}, and recall that the factor $2$ on the right-hand side was freely chosen by us.
A good substitute for this formula for the sets $C$ currently under consideration is given by
\begin{equation}
\label{4.1}
h_0(z) = (1+d_0-\mu_0)\, \frac{\sqrt{d_0-z}}{\mu_0-z} \prod_{j=1}^N \frac{\sqrt{(c_j-z)(d_j-z)}}{\mu_j-z} .
\end{equation}
This time, we opted for a factor of $1+d_0-\mu_0$, and this precaution is crucial because $\mu_0\in [-\infty,d_0]$ now comes from
an unbounded gap and we want our formula to stay well behaved when $\mu_0\to -\infty$. We continue to use \eqref{4.1} for $\mu_0=-\infty$ also,
and in this case, we of course interpret the factor in front of the product as simply $\sqrt{d_0-z}$.

As in Section 2, \eqref{4.1} can be proved by either carrying out the integration in \eqref{2.3} or, more conveniently perhaps, by observing that
the function defined by \eqref{4.1} is a Herglotz function that has the correct arguments on the real line. Finally, recall again that this requirement
of $h_0$ being a Herglotz function determines the choice of square root, so there is no ambiguity in \eqref{4.1}.

With this important adjustment in place, we can be very brief in the remainder of this section since the subsequent analysis will follow what we did in the previous
section closely.
We must make sure that the discussion of Section 2, which, on the surface at least, was based on \eqref{2.4} rather than \eqref{4.1}, still applies
to our current setting. Fortunately, there are no problems. In particular, we still have \eqref{2.11}, with the sum
taken over $0\le j\le N$ now. Notice that still $w_0=0$ if $\mu_0=-\infty$ or $\mu_0=d_0$, so again there is no point mass to assign to one of
the $m$ functions when $\mu_0$ takes one of these values and $s_0$ becomes irrelevant then. Moreover, $w_0$ also approaches zero
when $\mu_0\to-\infty$ (or $\mu_0\to d_0$, but this part is obvious); this makes sure that changing $s_0$ will not affect the canonical system
much when $\mu_0$ is close to $-\infty$, and interpreting $\widehat{\mu}_0$ as coming from a circle will thus again deliver the right topology.

The continuity of $h_0$ at $\mu_0=-\infty$ guarantees that there are no issues with the argument from the proof of Proposition \ref{P2.1}, and we
do have the analog of this result available for the Schr{\"o}dinger case also.
\begin{Theorem}
\label{T4.1}
The action of $\textrm{\rm PSL}(2,\mathbb R)$ on $\mathcal R_1(C)$ is fixed point free,
and every orbit contains a unique $H\in\mathcal S$.

The (equivariant) map $\textrm{\rm PSL}(2,\mathbb R)\times\mathcal S(C)\to\mathcal R_1(C)$, $(A, H) \mapsto A\cdot H$ is a homeomorphism.
\end{Theorem}
Recall that $\textrm{\rm PSL}(2,\mathbb R)$ is homeomorphic to a (non-compact) solid torus $\mathbb D\times\mathbb S^1$, as can be seen
from either the KAN decomposition or the polar representation of the elements of this group. See also \eqref{kan} below.
\begin{Corollary}
\label{C4.1}
(a) $\mathcal R_1(C)$ is homeomorphic to $\mathbb D\times \mathbb T^{N+1}$.

(b) The map $\mathcal R_1(C)/\PSL (2,\R)\to\mathcal S(C)$ that sends an orbit to its unique representative in $\mathcal S(C)$ is a homeomorphism,
and $\mathcal S(C)\cong\T^N$.
\end{Corollary}
As in the previous section, we do not really need the action of $\textrm{\rm PSL}(2,\mathbb R)$ to clarify the topology of $\mathcal R_1(C)$,
and indeed an easier way to identify this space with $\mathbb D\times\mathbb T^{N+1}$ is provided by the parametrization of Section 2.
\begin{proof}[Proof of Theorem \ref{T4.1} and Corollary \ref{C4.1} (sketch)]
As in the Dirac case, whether or not an $m$ function comes from a Schr{\"o}dinger equation is decided by the large $z$ asymptotics
of $m(z)$. The full criterion is awkward to state and use \cite{Lev,LevSar,Mar,RemdB}, but in the specialized situation under consideration, for
canonical systems from $\mathcal R_1(C)$, it simplifies considerably and boils down
to the following: we have $H\in\mathcal S$ if and only if the parameters satisfy $\mu_0=-\infty$, $A_+=0$, $D=1$ (and the remaining $\widehat{\mu}_j$,
if any, can take arbitrary values).

Now fix an $H\in\mathcal R_1(C)$. Inspection of \eqref{2.11} and \eqref{4.1} shows that if $\mu_0\not= -\infty$, then $m_+(z)$ has a holomorphic
continuation to a neighborhood of $(-\infty,\mu_0)$. Moreover, $m_+(-\infty):=\lim_{x\to -\infty} m_+(x)$ exists and $m_+(-\infty)\in\mathbb R$.
We can now act by a suitable
$B_1\in\textrm{\rm PSL}(2,\mathbb R)$ to make $B_1m_+(-\infty)=\infty$. So the new $m$ function $B_1\cdot m_+$ will now satisfy $\mu_0=-\infty$. An additional
action by an appropriate $B_2\in G$ from the dilation/translation subgroup from \eqref{3.5} will then move the parameters $A_+,D$ to the desired values
$A_+=0$, $D=1$, without changing $\mu_0=-\infty$ since $B_2\infty=\infty$ for such $B_2$. We have shown that the orbit of $H$
contains a Schr{\"o}dinger operator.

Uniqueness follows from similar arguments: If $H_j\in\mathcal S$ and $B\cdot H_1=H_2$, then, as just discussed, $B$ must fix $w=\infty$, so will
belong to $G$. Any non-identity element of $G$ changes at least one of the parameters $A_+,D$, so $B=1$ and in particular $H_2=H_1$. We have also
shown that the action is fixed point free; this latter claim is in fact completely trivial now because in the Schr{\"o}dinger case,
$\mathcal R_1(C)$ does not contain canonical systems with constant $m$ functions (unlike in the Dirac case),
and no $A\in\PSL(2,\R)$, $A\not= 1$, can fix more than one point when acting on $\C^+$.

With these adjustments in place, the rest of the argument proceeds as in the proof of Theorem \ref{T3.1}/Corollary \ref{C3.1}, and we leave
the remaining details to the reader.
\end{proof}
\begin{Theorem}
\label{T4.2}
(a) $\mathcal R_0(C)$ is homeomorphic to $\mathbb S^3$ if $N=0$.

(b) $\mathcal R_0(C)$ is not a manifold if $N\ge 1$. More precisely, a point $H\in\mathcal R_0(C)$ has a locally Euclidean neighborhood
if and only if $H\in\mathcal R_1(C)$, or, equivalently, if and only if $H\notin\mathcal Z$.
\end{Theorem}
If we assume the analog of Lemma \ref{L3.1}, then our original proof of Theorem \ref{T3.2} is still valid.
\section{Jacobi case: compact $C$}
Finally, we consider spectra $C$ with no unbounded component:
\[
C = \bigcup_{j=1}^{N+1} [d_{j-1},c_j], \quad\quad d_0<c_1<d_1<\ldots < c_{N+1} , \quad N\ge 0 .
\]
So $C$ now has two unbounded gaps $(-\infty, d_0)$, $(c_{N+1},\infty)$, in addition to the bounded gaps $(c_j,d_j)$, $j=1,2,\ldots ,N$.
As we will see, the presence of two parameters $\mu_0, \mu_{N+1}$ ranging over unbounded sets makes this case the most intricate one.
The issue is that now distinct $\mu_j$ can meet at $\mu=\infty$ while they were always well separated in the other cases.

Such compact spectra $C$ occur for Jacobi operators
\begin{equation}
\label{Jac}
(Jy)_n = a_n y_{n+1} + a_{n-1}y_{n-1} + b_n y_n ,
\end{equation}
acting on $y\in\ell^2(\mathbb Z)$. Using the device of singular intervals, these difference equations can also be rewritten as canonical systems
\cite[Sections 1.2, 5.3]{Rembook}. Here, a \textit{singular interval }of a canonical system $H$ is defined as a maximal interval $(a,b)$ with $H(x)=P_{\alpha}$ on $a<x<b$.
Recall that we denote by $P_{\alpha}$ the projection onto $v_{\alpha}=(\cos\alpha,\sin\alpha)^t$.
We refer to $\alpha$ (or, somewhat inconsistently but conveniently, sometimes also $v_{\alpha}$ or $P_{\alpha}$) as the \textit{type }of the singular interval.
As before, we will denote by $\mathcal J$ the collection of canonical systems that are, in this sense, equivalent to
Jacobi operators.

As our first assignment, we must again find a suitable version of \eqref{2.4}. We use the more intuitive notation $\mu_-=\mu_0\in [-\infty,d_0]$,
$\mu_+=\mu_{N+1}\in [c_{N+1},\infty]$ for the parameters from the unbounded gaps and then make the following choice for the
multiplicative constant:
\begin{multline}
\label{5.2}
h_0(z) =(1+d_0- \mu_-) (1+\mu_+-c_{N+1}) \times\\
\frac{\sqrt{(z-d_0)(z-c_{N+1})}}{(\mu_--z)(\mu_+-z)}
\prod_{j=1}^N \frac{\sqrt{(c_j-z)(d_j-z)}}{\mu_j-z}
\end{multline}
Of course, this formula can be proved in the same way as the previous versions. Note that this expression is continuous at points with $\mu_-=-\infty$
or $\mu_+=\infty$ if interpreted in the obvious way.

Usually, $h_0$ has a holomorphic continuation to a neighborhood of $z=\infty$. If specifically $\mu_-=-\infty$, $\mu_+=\infty$, then $h_0$ has a pole there,
and $h_0(z)=z +O(1)$. This implies that in this case, the representing measure $\nu$ from \eqref{2.31} has a point mass at infinity, $\nu(\{\infty\})=1$.
This in turn means that when implementing the procedure from Section 2, there is now an additional point mass that needs to be split between
$\nu_{\pm}$, but this time, it is not true that all of this must go into either $\nu_-$ or $\nu_+$. Recall that this requirement came
from the condition that $\sigma(H)\subseteq C$, but the presence or absence of a point mass at infinity will not affect the spectrum.
The upshot of all this is the following modification of \eqref{2.11} when $\pm\mu_{\pm}=\infty$:
\begin{equation}
\label{5.1}
m_+(z) = A_++D \left( \frac{1}{2}(h_0(z)-A)+ gz + \sum_{j=1}^N \frac{2s_j-1}{2} w_j \frac{1+\mu_j z}{\mu_j-z} \right) ,
\end{equation}
with $-1/2\le g\le 1/2$. If $(\mu_-,\mu_+)\not= (-\infty,\infty)$, then there is no additional parameter $g$ and \eqref{2.11} as written is valid.

The presence of $g$ some of the time seems to complicate matters considerably once we
start thinking about the proper topology on the parameter space.
We could try to view this space as a fibered space with base $\mathbb D\times\mathbb T^{N+2}$ and non-trivial fibers, consisting of the intervals $-1/2\le g\le 1/2$,
at the points $(A_+,D,\widehat{\mu}_j)$ satisfying $(\widehat{\mu}_-,\widehat{\mu}_+)=(-\infty,\infty)$. A different but closely related observation is that
if both $\mu_-\to-\infty$ and $\mu_+\to\infty$, then we no longer have $w_{\pm}\to 0$. This raises doubts about whether it is still appropriate to view
$\widehat{\mu}_{\pm}$ as coming from circles.

The general issue was analyzed in great detail in \cite{Remtop}, from the point of view of fibered spaces.
Adapted to our current situation, \cite[Corollary 1.6]{Remtop} suggests
that $\mathcal R_1(C)$ is still homeomorphic to $\mathbb D\times\mathbb T^{N+2}$.
So, loosely speaking, the fibers $-1/2\le g\le 1/2$ do not really stick out, but rather can be approximated by nearby points with trivial fibers.
The precise analysis was very tedious,
and we do not want to say anything about the details here, but let us point out that we can relate things quite directly to the situation studied in \cite{Remtop}
by using a map $F$ on canonical systems that transforms the $m$ functions as follows: $m_{\pm}(z)\mapsto m_{\pm}(-1/z)$. It is clear from the definitions
that $H\in\mathcal R_1(C)$ if and only if $F(H)\in\mathcal R_1(-1/C\cup\{ 0\})$. Moreover, $F$ is a homeomorphism between these spaces.

Spectra $C$ of Jacobi type become spectra $-1/C$ of Dirac type under this transformation, as studied in Section 3,
but with the added complication that the set now contains the isolated point $0$, which takes over the role
of one of the intervals $[b_{j-1},a_j]$. The treatment of \cite{Remtop} can be adapted to this situation.

However, all this is just background information, and we leave the matter at that. We will address (or perhaps bypass) this issue in a completely different way here,
by giving a prominent role to the $\PSL(2,\R)$ group action on $\mathcal R_1(C)$.

One more point needs our attention before we are ready to state the analog of Theorems \ref{T3.1}, \ref{T4.1}. A quick dimension count reveals that we cannot
really expect arbitrary orbits to intersect $\mathcal J$: For example, if $C=[-2,2]$,
then, as is well known \cite[Corollary 8.6]{Teschl}, $\mathcal J(C)$ contains only the free Jacobi matrix $a_n=1$, $b_n=0$.
Now we expect $\mathcal R_1(C)\cong\mathbb D\times\mathbb T^2$, which is a four-dimensional manifold,
corresponding to the parameters $A_+$, $D$, $\widehat{\mu}_-$, $\widehat{\mu}_+$. On the other hand,
the acting group $\PSL(2,\R)\cong \mathbb D\times\mathbb S^1$ is only three-dimensional. The situation for general $C$ is similar: the discrepancy
between $\mathcal R_1(C)$ and $\PSL(2,\R)\times\mathcal J(C)$ seems to be one circle (equivalently, one dimension in the torus factor).

We need a wider class of representatives, and for now we offer several options. Eventually (in Lemma \ref{L5.1} below) the choice will have to be made
very carefully.

When a Jacobi equation is rewritten as a canonical system,
then $H$ consists of singular intervals $H(x)=P_{\alpha}$, $a<x<b$, only. The origin $x=0$ is an endpoint, and the first singular interval to the left
is $(-1/a_0^2,0)$, with $a_0>0$ being one of the coefficients from \eqref{Jac}, and its type is $e_2$.
So $H(x)=P_{e_2}=\bigl(\begin{smallmatrix} 0 & 0\\0 & 1\end{smallmatrix}\bigr)$
on this interval $-1/a_0^2<x<0$. Please see \cite[Section 5.3]{Rembook}
for these statements.

Recall also that whether or not a Herglotz function is the $m$ function of a Jacobi matrix can again be decided
by looking at the large $z$ asymptotics. See \cite[Theorem 2.31]{Teschl}.
The $m$ functions $m(z)=m_{\pm}(z;H)$, $H\in\mathcal R_1(C)$, currently under consideration are guaranteed to be holomorphic at $z=\infty$ or have a
first order pole there; compare \eqref{5.12} below. In this case, these criteria take the following form:
\begin{enumerate}
\item $m(z)$ is the $m$ function $m(z)=m_+(z;J)$ of a right half line Jacobi matrix $J$
if and only if $m(z)=-1/z +O(1/z^2)$.
\item $m(z)$ is a \textit{left }half line Jacobi matrix $m$ function if and only if $m(z)=bz +O(1)$, $b>0$. 
\end{enumerate}
\begin{Theorem}
\label{T5.1}
The action of $\PSL(2,\R)$ on $\mathcal R_1(C)$ is fixed point free, and
every orbit contains a unique $H$ of each of the following types:\\
(a) $H(x) = H_0(x-t/a_0^2)$ with $H_0\in\mathcal J(C)$, $0\le t<1$; these latter quantities $H_0,t$ are also unique;\\
(b) $m_+(z;H)= -1/z + O(1/z^3)$
\end{Theorem}
In part (a), $a_0>0$ again refers to the Jacobi coefficient of the Jacobi matrix that is associated with $H_0$. In particular, this quantity is determined by $H_0\in\mathcal J$.
In part (b), the chosen representative is a Jacobi matrix on the right half line, and in fact a special one, with $b_1=0$, but there is no control on what
happens on the left half line.
\begin{proof}
The first claim, about the action being fixed point free, is again trivial because $\mathcal R_1(C)$ does not contain canonical systems with constant $m$ functions
in the Jacobi case and no non-identity group element can fix more than one point when acting on $\C^+$.

(a) Let $H\in\mathcal R_1(C)$. Close inspection of \eqref{2.11}, \eqref{5.2}, \eqref{5.1} shows that $m_+(z)=m_+(z;H)$ is holomorphic at $z=\infty$ or has a pole there;
more precisely still,
\begin{equation}
\label{5.12}
m_+(z)=b_0z+a+\frac{c}{z}+O(1/z^2) ,
\end{equation}
with $b_0\ge 0$, $a\in\mathbb R$, $c<0$. The inequalities follow from the Herglotz property of $m_+$:
Clearly, $m_+$ could not satisfy $\Im m_+(z)>0$ for all (large) $z\in\C^+$ if we did not have $b_0\ge 0$. As for $c$, we observe that the Herglotz representation
of $m_+$ implies that $m_+(z)-b_0z-a$ also is a Herglotz function, which in our current situation is holomorphic at $z=\infty$. Again, by looking at large $z$,
we see that such a function can be a Herglotz function only if its Taylor expansion starts with $c/z+\ldots$, $c<0$.

If $b_0=0$ in \eqref{5.12}, then we start out by acting by the combined translation and inversion
\[
\begin{pmatrix} 0 & -1\\ 1 & 0 \end{pmatrix} \cdot \begin{pmatrix} 1 & -a \\ 0 & 1 \end{pmatrix}\cdot m_+(z) = \frac{-1}{m_+(z)-a}
\]
to reach another point in the same orbit with $b_0>0$ now. If $b_0>0$ initially, then we skip this first step.

We follow up by a suitable translation and dilation $g\in G$ to obtain
a new canonical system $H_1= B\cdot H$ whose $m$ function satisfies
\begin{equation}
\label{5.3}
m_+(z; H_1) = bz-\frac{1}{z}+O(1/z^2)\equiv bz + m_0(z), \quad\quad b> 0 .
\end{equation}
As we discussed, $m_0$ is the $m$ function of a (right) half line Jacobi matrix, which we can write as a canonical system $H_2(x)$, $x\ge 0$.
Moreover, an extra term $bz$ in the $m$ function corresponds to the insertion of an initial singular interval $H=P_{e_2}$ of length $b$;
compare \cite[Theorem 5.19]{Rembook} and its proof. So the right half line of $H_1=B\cdot H$ is given by
\[
H_1(x) = \begin{cases} P_{e_2} & 0< x<b \\ H_2(x-b) & x>b
\end{cases} .
\]
We can rephrase what we have done so far as follows: (1) The shifted canonical system $H_0(x)=H_1(x+b)$ has the (Jacobi) $m$ function $m_0$ as its
right half line $m$ function $m_+(z;H_0)$; (2) The left half line of $H_0$ starts with a singular interval $H_0(x)=P_{e_2}$, $-L<x<0$, $L\ge b>0$.

Now (2) implies that $m_-(z;H_0)=Lz + O(1)$, so the criterion reviewed above makes sure that the left half line is
a Jacobi matrix also, that is, $H_0\in\mathcal J$.
We have found a shifted version $H_1(x)=H_0(x-b)$ of an $H_0\in\mathcal J$ in the orbit, as desired.
Reviewing one more time how exactly we obtained this system $H_1$, we can also confirm that we did not have to shift $H_0\in\mathcal J$ by
more than the length of its first singular interval $(-1/a_0^2,0)$ on the left half line, so $b=t/a_0^2$ with $0\le t\le 1$, as claimed.
In this whole argument, we also use the (easy) fact that $\mathcal R_1(C)$ is invariant under shifts; compare \cite[Theorem 7.9(a)]{Rembook}.

This almost proves the existence part.
It remains to discuss why we never need to shift by the full length of this interval, corresponding to $t=1$ in the statement of Theorem \ref{T5.1}(a).
If we took $t=1$, then the resulting coefficient function $H(x)=H_0(x-1/a_0^2)$
will have singular intervals $(-L,0)$, $(0,1/a_0^2)$ of types $P_{\alpha}$ and $P_{e_2}$, $v_{\alpha}\not= e_2$, respectively, near $x=0$.
We can act on this by a suitable matrix of the form
\begin{equation}
\label{5.4}
A = \begin{pmatrix} 1 & 0 \\ a & 1 \end{pmatrix}\begin{pmatrix} 0 & -1 \\ 1 & 0 \end{pmatrix}
\end{equation}
to change the types to $P_{e_2}$ and $P_{e_1}$, respectively. The value of $a$ that is needed
will of course depend on $\alpha$. Recall here that the action by an $A\in\PSL(2,\R)$ changes the coefficient function
to $H_1(x)=A^{-1t}H(x)A^{-1}$ \cite[Theorem 3.20]{Rembook}. Then a further transformation by a suitable
\begin{equation}
\label{5.5}
B = \begin{pmatrix} c & 0 \\ 0 & 1/c \end{pmatrix} , \quad c>0,
\end{equation}
will not change these types $e_1,e_2$ but will allow us to adjust the lengths of the singular intervals to reach a situation where $H(x)=P_{e_1}$ on exactly $0<x<1$.
This happens because we will need to run a change of variable to keep our coefficient functions trace normed. Now the connection between
initial singular intervals and large $z$ asymptotics of the $m$ functions will imply that $B\cdot A\cdot H\in\mathcal J$.
See \cite[Theorem 4.34]{Rembook} and its proof for this final step; Proposition \ref{P5.1} below and the brief discussion that follows
may also be helpful in this context. To sum this up, we have shown that if $H_0\in\mathcal J(C)$, then the orbit of $H(x)=H_0(x-1/a_0^2)$ contains
an $H_1\in\mathcal J(C)$, so we are indeed never forced to take $t=1$ in the statement of Theorem \ref{T5.1}(a).

Moving on to the uniqueness part, we observe that if $H$ is of the form specified (a shifted $H_0\in\mathcal J$), then from the types of the singular intervals near
$x=0$, we know that the half line $m$ functions have the asymptotics $m_-(z)=bz + O(1)$, $b>0$, and $m_+(z) = cz -1/z +O(1/z^2)$, $c\ge 0$.

If we now act on such an $H$ by an $A\in\PSL(2,\R)$, then we can obtain another $H_1=A\cdot H$ of the same type only if $A\infty=\infty$, that is,
$A\in G$, because otherwise we would destroy the required asymptotics of $m_-$. However, now an $A\in G$ will preserve the asymptotics of $m_+$ only if $A=1$.
(The small detail that $b>0$ rather than only $b\ge 0$ is known here was crucial to the argument, which would otherwise break down, as it must, since
we could then act by an $A$ that switches $0,\infty$, such as the inversion $J$. We do know that $b>0$ because we agreed not to shift by the full
length of the singular interval $(-1/a_0^2,0)$.)

Finally, we discuss the uniqueness claim about the parameters $J,t$. We know that
$H_0(x-t/a_0^2)$ has a singular interval $I$ of type $e_2$ with $0\in I$ (or possibly $0$ is the right endpoint, if $t=0$), and we can recover $t$ simply
by checking what fraction of its total length $1/a_0^2$ lies to the right of $0$. Then we can shift back and recover $J$ uniquely since Jacobi matrices
are determined by their $m$ functions.

(b) As we discussed previously, if $H\in\mathcal R_1(C)$ is given and we don't do anything, then we already have $m_+(z)=bz+O(1)$, $b>0$, or
$m_+(z) = a + c/z+ O(1/z^2)$, $c<0$, for large $z$. In the first case, we can act first by a translation $T$ to improve this to $T\cdot m_+(z)=bz+O(1/z)$,
and then the inversion $J=\bigl( \begin{smallmatrix} 0 & -1 \\ 1 & 0 \end{smallmatrix}\bigr)$, followed by the multiplication $B$ by $b>0$, will produce
the desired asymptotics $B\cdot J\cdot T\cdot m_+=-1/z+O(1/z^3)$. The second case can be reduced to the first one by starting out with a translation
plus inversion.

Uniqueness follows from the by now familiar arguments: If $H_1,H_2$ both have $m_+$ functions of the type described and $H_2=A\cdot H_1$,
then clearly $A0=0$, or else the large $z$ asymptotics would not be preserved. So $A$ is the transpose of a matrix from $G$, and then it's easy to see that in fact $A=1$.
\end{proof}
In this proof of Theorem \ref{T5.1}, we have focused on $m$ functions, their asymptotics, and the effect of the group action on these.
We also could have worked with the coefficient functions $H(x)$ directly, and it is perhaps worthwhile to indicate this briefly since it sheds some additional light
on the whole argument.
\begin{Proposition}
\label{P5.1}
Every $H\in\mathcal R_1(C)$ consists of singular intervals only: $H(x)=P_{\alpha_j}$, $a_{j-1}<x<a_j$. These don't accumulate anywhere,
that is,
\[
\ldots <a_{-1}<a_0<a_1<\ldots, \quad\lim_{j\to -\infty} a_j = -\infty, \lim_{j\to\infty}a_j=\infty .
\]
\end{Proposition}
For us here, this result, which may be of some independent interest, can be viewed as an immediate consequence of Theorem \ref{T5.1}(a) because an $H\in\mathcal J(C)$
is of the type described in Proposition \ref{P5.1},
and then neither shifts nor the group action change the general structure of $H(x)$. But the result could also be proved directly
without too much effort, and then an alternative proof of Theorem \ref{T5.1} could be based on it. We then first shift a general $H\in\mathcal R_1(C)$ to
move one of the endpoints of the singular intervals to $x=0$, and then we act by a suitable $A\in\PSL(2,\R)$ to reach the types $e_2$ and $e_1$, respectively,
for the two intervals adjacent to $x=0$, and finally we act by a dilation to make the length of first singular interval on the right half line equal to $1$.
These are exactly the properties needed to ensure that the canonical system lies in $\mathcal J$, and this can be confirmed by using the fact that initial
singular intervals and their types contribute the leading terms to the large $z$ asymptotics of $m_{\pm}(z)$. See again \cite[Theorem 4.34]{Rembook} and its proof
for further details on this.

To state and prove the analog of the remaining parts of Theorems \ref{T3.1}, \ref{T4.1}, we need an additional tool.
The (left) \textit{shift }$S$ on Jacobi matrices is defined in the obvious way: if $J$ has coefficients $a_n,b_n$, $n\in\Z$, then
the coefficients of $SJ$ are $a_{n+1},b_{n+1}$. This map $S$ preserves spectra and the property of being reflectionless.
We can think of it as acting on the corresponding space $\mathcal J(C)$ of canonical systems, and then it becomes a homeomorphism.
We will occasionally be quite cavalier about the distinction between Jacobi matrices $J$ and the associated canonical systems $H=H_J\in\mathcal J$ in the sequel.

Ignoring that policy for now, we may reasonably ask ourselves what exactly we need to do to a canonical system
$H\in\mathcal J$ to implement the map $S$. The trap to avoid is to think that this is also just a shift of $H(x)$.
In fact, that can not possibly work since the types of the singular intervals near $x=0$ will no longer be right after a shift.
Rather, we need to follow up the shift of $H(x)$ with the action of a suitable $A\in\PSL(2,\R)$. Such combined maps are called \textit{twisted shifts.}

Let us give the detailed statement for the right shift $S^{-1}$. We already know, from the proof of Theorem \ref{T5.1}(a),
that if $H\in\mathcal J$, then there is a unique $A=A(H)\in\PSL(2,\R)$ such that
$A\cdot H(x-1/a_0^2)\in\mathcal J$ also. On the other hand, a computation that compares the transfer matrices of
the Jacobi matrix and the canonical system shows that if we set
\[
A=A(H) = \begin{pmatrix} 0 & -1/a_0 \\ a_0 & -b_0/a_0 \end{pmatrix},
\]
then
\begin{equation}
\label{5.9}
A(H)\cdot H\left( x-\frac{1}{a_0^2}\right) =H_{S^{-1}J}(x) ,
\end{equation}
the canonical system corresponding to the right shifted version $S^{-1}J$
of the Jacobi matrix $J$ that was associated with the original $H\in\mathcal J$.
This gives a simple explicit formula for $A(H)$, $H\in\mathcal J$, in terms of the Jacobi coefficients, but actually we won't use this in the sequel.
What matters for us is the fact that $A(H)$ depends continuously on $H\in\mathcal J(C)$, and this information can also be conveniently extracted
from the discussion above that constructed $A(H)$ as the product of the matrices from \eqref{5.4}, \eqref{5.5}.
We conclude this short digression on twisted shifts and refer the reader to \cite[Section 7.1]{Rembook} for further background and also to \cite{RemToda},
where the terminology of twisted shifts was introduced and their usefulness advertised.

The key fact about the shift $S$ on $\mathcal J(C)$ is that it can be embedded in a continuous flow $\phi_t: \mathcal J(C) \to \mathcal J(C)$.
So $\phi_0=\operatorname{id}$, $\phi_1=S$, $\phi_{s+t}=\phi_s\phi_t$. Moreover, $\phi_t$ for any $t\in\R$ is a homeomorphism,
and the map $(t,J)\mapsto \phi_tJ$ is also continuous.

A convenient way to obtain such a $\phi_t$
is to use a suitable flow from the \textit{Toda hierarchy. }Recall that these flows commute with the shift, and they can be linearized simultaneously
on $\mathcal J(C)\cong\T^N$. See \cite[Ch.\ 13]{Teschl}, especially Theorem 13.5 there. This means that if $\T^N$ is given suitable coordinates
$(x_1, \ldots , x_N) \in\R^N$, with $x,x'$ representing the same point if and only if $x\equiv x'\bmod 1$, then $Sx = x+a$, $\psi_t x = x+tb$, and any $b\in\R^N$
is available here by picking a suitable flow from the hierarchy. For our purposes, we of course need $b=a$, and we fix such a flow once and for all and
denote it by $\phi_t$, as above. It is perhaps also worth pointing out that the hierarchy of Toda flows is available for Jacobi matrices only; for general canonical
systems, no analog is currently known, and this is in fact an issue that seems to deserve closer investigation. First steps were taken in \cite{HurOng,RemToda}.

As a final preparation, we now need the following variation on Theorem \ref{T5.1}(a). We denote the canonical system corresponding to a Jacobi matrix $J$
by $H_J\in\mathcal J$ when needed, but recall again that we don't always carefully distinguish between $J$ and $H_J$.

We turn on the transformation from \eqref{5.9} gradually, but using the alternative construction from \eqref{5.4}, \eqref{5.5}. So let
\[
A(t,J) = \begin{pmatrix} 1+t(c-1) & 0 \\ 0& \frac{1}{1+t(c-1)} \end{pmatrix} \begin{pmatrix} 1 & 0\\ ta &1\end{pmatrix}
\begin{pmatrix} \cos\pi t/2 & -\sin\pi t/2\\ \sin\pi t/2&\cos\pi t/2\end{pmatrix}  ,
\]
with $a\in\R, c>0$ being the parameters needed for $J$.
\begin{Lemma}
\label{L5.1}
Every orbit contains a unique canonical system of the form
\begin{equation}
\label{5.11}
A(t,SJ)\cdot H_{\phi_t J} \left( x - \frac{t}{a_0^2(\phi_t J)} \right) ,
\end{equation}
with $J\in\mathcal J(C)$, $0\le t<1$; these latter quantities $J,t$ are also uniquely determined by the orbit.
\end{Lemma}
This looks rather unwieldy, but closer inspection reveals one highly desirable feature: if we denote the canonical system from \eqref{5.11} by $H(t,J)$,
then $H(1,J)=H(0,J)$. Indeed, by construction we have $\phi_1J=SJ$, $A(1,SJ)=A(SJ)$, so \eqref{5.9} shows that
\[
H(1,J)=A(SJ)\cdot H_{SJ} \left( x-\frac{1}{a_0^2(SJ)}\right) =H_J(x) = H(0,J) .
\]
This property will become crucial because, as we'll see, the extra parameter $t$ will only be compatible with the topology of $\mathcal R_1(C)$
if it can be interpreted as coming from a circle.
\begin{proof}
To prove that the orbit of a given $H\in\mathcal R_1(C)$
contains a canonical system of the form \eqref{5.11}, we can of course ignore the action of $A(t,SJ)\in\PSL(2,\R)$, which
will leave us in the same orbit. By Theorem \ref{T5.1}(a), the orbit under consideration contains a canonical system
of the form $H_1(x)=H_{J_0}(x-t/a_0^2(J_0))$, with $0\le t<1$, $J_0\in\mathcal J(C)$. If we put $J=\phi_{-t}J_0$, we see that $H_1$ is of the required type.

This last step can be reversed, and thus we see that we have a bijection between the collection of systems $\{H_J(x-t/a_0^2(J))\}$ from Theorem \ref{T5.1}(a)
and the collection $\{ H_{\phi_t J} (x - t/a_0^2(\phi_t J))\}$ from \eqref{5.11}, implemented by mapping $t\mapsto t$ and $J\mapsto\phi_{-t}J$.
As a consequence, uniqueness now follows from what we already did in the proof of Theorem \ref{T5.1}(a). Finally, note that $t,J$ can be reconstructed
from $H_{\phi_tJ}(x-t/a_0^2(\phi_tJ))$ in the same way as before.
\end{proof}
\begin{Theorem}
\label{T5.2}
We have $\textrm{\rm PSL}(2,\mathbb R)\times\mathbb S^1\times \mathcal J(C)\cong\mathcal R_1(C)$, and a homeomorphism
is provided by the map
\begin{equation}
\label{5.24}
(B,e^{2\pi it}, J) \mapsto B\cdot A(t,SJ)\cdot H_{\phi_t J} \left( x - \frac{t}{a_0^2(\phi_t J)} \right) .
\end{equation}
\end{Theorem}
As is probably already clear from what we did above, it is understood here that we use the representative of $t$ with $0\le t<1$ on the right-hand side.
\begin{Corollary}
\label{C5.1}
$\mathcal R_1(C)$ is homeomorphic to $\mathbb D\times\T^{N+2}$, and
\[
\mathcal R_1(C)/\PSL(2,\R)\cong\mathbb S^1\times\mathcal J(C)\cong\mathbb S^1\times\T^N\cong\T^{N+1} .
\]
\end{Corollary}
\begin{proof}[Proof of Theorem \ref{T5.2}]
The previous work has established that this map
\begin{equation}
\label{5.51}
F: \textrm{\rm PSL}(2,\mathbb R)\times\mathbb S^1\times \mathcal J(C) \to\mathcal R_1(C)
\end{equation}
which acts as described in \eqref{5.24} is a bijection. It is also continuous; this is mostly obvious by inspection, the only issue being
the points with $t=0$, but here we refer to the discussion following the statement of Lemma \ref{L5.1}.

We now want to use an automatic continuity result to deduce that $F$ is a homeomorphism. To do this, we identify
$\PSL(2,\R)$ with $\C^+\times\mathbb S^1$, using the KAN decomposition: We send a point $(a+ic,e^{2i\alpha})\in\C^+\times\mathbb S^1$
to the group element represented by the matrix
\begin{equation}
\label{kan}
\begin{pmatrix} c & a/c \\ 0 & 1/c \end{pmatrix} \begin{pmatrix} \cos \alpha & -\sin\alpha\\\sin\alpha&\cos\alpha \end{pmatrix} ;
\end{equation}
this sets up a homeomorphism between $\C^+\times\mathbb S^1$ and $\PSL(2,\R)$.

We now proceed as in Lemma \ref{L3.1}. We extend $F$ in such a way that the induced map on a suitable
quotient delivers a homeomorphism to $\mathcal R_0(C)$. Taking the original treatment as our guideline, it is fairly obvious how we want to do this.
Reinterpret $F$ as a map
\[
F: \C^+\times\mathbb S^1\times\mathbb S^1\times\mathcal J(C)\to\mathcal R_1(C) ,
\]
using the above identification of $\PSL(2,\R)$ with the product of the first two factors, and then extend
\[
F: \overline{\C^+}\times\mathbb S^1\times\mathbb S^1\times\mathcal J(C)\to\mathcal R_0(C)
\]
by setting $F(a,w,z,J)=K_a\in\mathcal Z$, the singular system with $\pm m_{\pm}(z)\equiv a$, for $a\in\R^{\infty}$.

We claim that this extended map $F$ is still continuous. Of course, we only need to verify continuity at the added points
$(a,w,z,J)$, $a\in\R^{\infty}$, since these form a closed subset. We can then argue as in the proof of Proposition \ref{P2.1}, so we will be content with giving a sketch.
Fix such a point and assume that $(u_n,w_n,z_n,J_n)\to (a,w,z,J)$. We can focus on the $u_n\notin\R^{\infty}$ here (if any) because
what we are trying to show is already obvious for the other points.

The $m$ functions $m_n(z)\equiv m_+(z; F(u_n,w_n,z_n,J_n))$ are then of the form
\[
m_n(z) = c_n^2 M_n(z) + a_n , \quad u_n= a_n+ic_n ,
\]
and here the $M_n$ similarly are the $m$ functions of $F(i,w_n,z_n,J_n)$. The key observation is that these only depend on $(w_n,z_n,J_n)$, and these latter parameters
come from the \textit{compact }space $\mathbb S^1\times\mathbb S^1\times\mathcal J(C)$. Here, we use the fact that indeed $\mathcal J(C)\cong\T^N$
is compact, which is well known
and easily established, using the parameters $\widehat{\mu}_j$, $j=1,\ldots ,N$, to represent $\mathcal J(C)$. We tacitly used this already
above in our brief review of Toda flows. See also \cite[Theorem 1.5]{PoltRem2}.

As a consequence, we have uniform control on, say, $M_n(i)$, which can only vary over a compact subset of $\C^+$.
This step is the analog of \eqref{2.41}. Given this, we can now finish the proof of the continuity
of $F$ at $(a,w,z,J)$ as in the proof of Proposition \ref{P2.1}.

Finally, the induced map
\[
F_1 : \overline{\C^+}\times\mathbb S^1\times\mathbb S^1\times\mathcal J(C)/\!\!\sim\: \to\mathcal R_0(C)
\]
on the quotient by the equivalence relation
\[
(a, w,z,J)\sim (a,w',z',J') ,\quad a\in\R^{\infty} ,
\]
is a homeomorphism. This follows because the map is a bijection, by construction and what we already proved about the original map \eqref{5.51}.
Moreover, we just established continuity, and since we are mapping between compact metric spaces now, the continuity of the inverse is
automatic. This then also implies that the map from \eqref{5.51} is a homeomorphism since the original smaller spaces are embedded as (open)
subspaces in the larger compact spaces.
\end{proof}
\begin{Theorem}
\label{T5.3}
$\mathcal R_0(C)$ is not a manifold. More precisely, a point $H\in\mathcal R_0(C)$ has a locally Euclidean neighborhood
if and only if $H\in\mathcal R_1(C)$, or, equivalently, if and only if $H\notin\mathcal Z$.
\end{Theorem}
The proof of Theorem \ref{T3.2}(b) still applies here, given the identification of $\mathcal R_0(C)$ from the last part of the proof of Theorem \ref{T5.2}.

\end{document}